\documentclass[a4paper,11pt]{amsart}

\usepackage{amsthm, amsmath, latexsym, amsfonts, epsfig}
\usepackage[all]{xy}
\usepackage{amssymb}
\usepackage{mathrsfs}
\input xy
\xyoption{all}
\newcounter{teorema}
\newtheorem{teo}[teorema]{Theorem}
\newtheorem{Lemma}{Lemma}[section]
\newtheorem{Proposition}[Lemma]{Proposition}
\newtheorem{Theorem}[Lemma]{Theorem}
\newtheorem{Remark}[Lemma]{{Remark}}
\newtheorem{Corollary}[Lemma]{Corollary}
\newtheorem{Definition}[Lemma]{Definition}

\newtheorem{Example}[Lemma]{Example}
\newtheorem{Notation}[Lemma]{Notation}
\theoremstyle{remark}

\newcommand{\bP}{\mathbb{P}}

\title{Notes on projective normality of reducible curves}

\author{Edoardo Ballico}
\address{Dept. of Mathematics\\
 University of Trento\\
38123 Povo (TN), Italy}
\email{ballico@science.unitn.it}
\thanks{The author was partially supported by MIUR and GNSAGA of INdAM (Italy).}

\author{Silvia Brannetti}
\address{Dept. of Mathematics\\
University of Roma Tre\\
Largo S. Leonardo Murialdo 1,
00146 Rome, Italy}
\email{brannett@mat.uniroma3.it}


\begin{document}

\begin{abstract}
We give some results on quadratic normality of reducible curves canonically embedded and partially extend this study to their projective normality.
\end{abstract}

\maketitle

\section*{Introduction}

Let $C$ be a smooth curve of genus $g$ over an algebraically closed field $k$. The canonical bundle $\omega_C$ induces an embedding of $C$ in $\bP^{g-1}$ if and only if $C$ is not hyperelliptic; we indicate the power $\omega_C^{\otimes n}$ by $\omega_C^n$ for any $n\in\mathbb N$. One says that $C$ is \textit{projectively normal} if the maps
\begin{equation}\label{projnorm}
H^0(\bP^{g-1},\mathcal O_{\bP^{g-1}}(k))\rightarrow H^0(C,\omega_C^k)
\end{equation}
are surjective for every $k\geq 1$.
In other words,
$C$ is projectively normal if and only if the hypersurfaces of degree $k$ in $\bP^{g-1}$ cut a complete linear series on $C$ for any $k$. If $k=1$ and the map (\ref{projnorm}) is surjective, we say that $C$ is \textit{linearly normal}, which means that the curve is embedded via a complete linear series. If $\omega _C$ is ample, then an equivalent formulation states that $C$ is projectively normal if the maps
\begin{equation}\label{projnorm2}
{\rm Sym}^k H^0(C,\omega_C)\rightarrow H^0(C,\omega_C^k)
\end{equation}
are surjective for every $k\geq 1$, because the surjectivity of all these maps when $\omega _C$ is ample implies the very ampleness of $\omega _C$.

If $C$ is a smooth, non-hyperelliptic curve, Castelnuovo and Noether proved that its canonical model is projectively normal (see \cite{ACGH}). When we deal with singular curves, though, the problem becomes harder: for integral curves, in \cite{Kleiman} the authors generalize Castelnuovo's approach proving that linear normality is equivalent to projective normality. For reducible curves yet not much is known: properties of the canonical map for Gorenstein curves, i.e. the map induced by the dualising sheaf, are investigated in \cite{cfhr}, whereas in \cite{Franciosi2} the author gives a sufficient condition for line bundles on non-reduced curves to be \textit{normally generated} (see \ref{normGen}). The projective normality of reducible curves is studied in \cite{Schreyer}; more in general, since the problem of studying projective normality reduces to the study of multiplication maps, we refer to \cite{Ballico} and \cite{Franciosi2} for these items.

In this paper we investigate the projective normality of reducible curves restricting the problem to suitable subcurves. The first step is to study the \textit{quadratic normality}, i.e. the surjectivity of the maps in (\ref{projnorm}) for $k=2$. Let $X$ be a connected, reduced and Gorenstein projective curve of genus $g$ with $\omega_X$ very ample. Assume that $X$ has planar singularities at the points lying on at least two irreducible components. Our main result about quadratic normality is the following theorem.

\begin{teo}\label{teorema}
Let $X$ be a curve as above, and set $X=A\cup B$ with $A,B$ connected subcurves being smooth at $D:=A\cap B$. If $A\neq\emptyset$ and the map $$\mu_{\omega_A,{\omega_X}|_A}:H^0(A,\omega_A)\otimes H^0(A,{\omega_X}|_A)\rightarrow H^0(X,\omega_A\otimes {\omega_X}|_A)$$ is surjective, then $X$ is quadratically normal.
\end{teo}

We also study certain multiplication maps in order to establish sufficient conditions that imply the surjectivity of the map in (\ref{projnorm2}) for some $k$ (\textit{$k$-normal generation}) assuming to know the surjectivity for $(k-1)$ (see Proposition \ref{kNormGen}).

We divided the paper in two sections: in the first one we show our results about multiplication maps of reducible curves and apply them to the study of quadratic normality and of $k$-normal generation given the $(k-1)$-normal generation of the canonical bundle. In the second section we show some applications of our results to interesting cases and give some examples.
\vspace{.5cm}

\emph{Acknowledgements.}

We wish to thank Ciro Ciliberto, Claudio Fontanari and Marco Franciosi for precious conversations.

\section{Quadratic normality}

For any reduced projective curve $X$ and any line bundles $M,N$ on $X$ let
\begin{equation}
\mu_{M,N}:H^0(X,M)\otimes H^0(X,N)\longrightarrow H^0(X,M\otimes N);
\end{equation}
denote the multiplication map. Set $\mu_M=\mu_{M,M}$.
Given the dualizing sheaf $\omega_X$ on $X$, we are interested in studying the surjectivity of the map $\mu_{\omega_X}$. In particular, when we assume that $X$ is canonically embedded this is equivalent to saying that $X$ is \textit{quadratically normal}. We have

\begin{Proposition}\label{teo1}
Let $X$ be a connected reduced curve of genus $g$ with planar singularities and $\omega_X$ very ample. Assume that $X=A\cup B$, with $A,B$ connected and smooth at $D:=A\cap B$. If
\begin{itemize}
\item[(i)]$\mu_{\omega_A,{\omega_X}|_A}$ is surjective,
\item[(ii)]$\mu_{{\omega_X}|_B}$ is surjective,
\end{itemize}
then
$\mu_{\omega_X}$ is surjective.
\end{Proposition}

In order to prove the proposition, we need some background material. We are going to keep the notation used in the statement of Proposition \ref{teo1}. Let $D:= A\cap B$ be the scheme-theoretic intersection. We will view $D$ also as a subscheme
of $A$ and $B$. Since both $A$ and $B$ are smooth at each point of the support of $D$, that we denote by $supp(D)$, the scheme $D$ is a Cartier divisor of both $A$ and $B$; more in general, this is true if $X$ has only planar singularities at each point
of $supp(D)$, because in this case a local equation of $B$ in an ambient germ of a smooth surface gives a local equation of $D$ as a subscheme of $A$.

\begin{Remark}\label{properties}
{\rm According to the notation above, we have that
\begin{itemize}
\item[(i)] It is well known that a curve with planar singularities is Gorenstein.
\item[(ii)] Since $X$ is Gorenstein and locally planar at the points of $supp(D)$, then $A$ and $B$ are Gorenstein as well, so that $\omega_A$ and $\omega_B$ are both line bundles on $A$ and $B$.

\item[(iii)] Since $X$ is locally planar at the points of $supp(D)$, the adjunction formula gives ${\omega_X}|_A=\omega_A(D)$ and ${\omega_X}|_B=\omega_B(D)$. Thus
$\deg  ({\omega_X}|_A) =2g_A-2+\delta$ and $\deg ({\omega_X}|_B) = 2g_B-2+\delta$, where of course $g_A, g_B$ are the arithmetic genera of $A$ and $B$, and $\delta=\deg (D)$.  \end{itemize} }
\end{Remark}

\begin{Lemma}\label{a1}
Let $Z$ be a reduced, Gorenstein and connected projective curve. Let $E$ be an effective Cartier divisor on $Z$ such that $E \neq 0$.
Then $h^0(\mathcal {I}_E)=0$ and $h^1(\omega _Z(E)) =0$.
\end{Lemma}

\begin{proof}
Since $Z$ is connected, $h^0(\mathcal {O}_Z)=1$. Since $E$ is effective and non-empty, we get $h^0(\mathcal {I}_E)=0$. We apply the duality
for locally Cohen-Macaulay schemes, i.e. we apply to the scheme $X:=Z$ and the sheaf $F:= \omega _Z(E)$ the case $r=p=1$
of the theorem at page 1 of \cite{ak}. We get $h^1(\omega _Z(E)) = \dim (Ext ^0(\omega _Z(E),\omega _Z))$, i.e.
$h^1(\omega _Z(E)) =h^0({\it Hom}(\omega_Z(E),\omega _Z))$. Since $\omega _Z$ is assumed to be locally free, we get
$h^1(\omega _Z(E)) = h^0({\it Hom}(\mathcal {O}_Z(E),\mathcal {O}_Z))=0$.
\end{proof}

\begin{Lemma}\label{lem1}
Let $X$ be a connected reduced curve of genus $g$ with planar singularities and $\omega_X$ very ample. Assume that $X=A\cup B$, with $A,B$ connected and smooth at $D:=A\cap B$. For any subcurve $Z$ of $X$ we consider the map
$$\rho_Z:H^0(X,\omega_X)\longrightarrow H^0(Z,{\omega_X}|_Z).$$
Then $\rho_A$ and $\rho_B$ are surjective.
\end{Lemma}

\begin{proof}
To fix ideas we work on $Z=A$; let us consider the exact sequence:
$$0\rightarrow\mathcal I_A\otimes\omega_X\rightarrow\omega_X\rightarrow{\omega_X}|_A\rightarrow 0.$$
We claim that $\mathcal I_A\otimes\omega_X=\omega_B$. To prove this, we notice that since $X$ has only planar singularities, it can be embedded in a smooth surface $S$, where $X$, $A$ and $B$ are Cartier divisors. Thus $D$ is a Cartier divisor
of $A$ and of $B$ (but seldom of $X$). By the adjunction formula we have that
$$\omega_X={\omega_S(A+B)}|_X,$$ then
$$\omega_B={\omega_S(B)}|_B={\omega_S(A+B-A)}|_B=({\omega_S(A+B-A)}|_X)|_B$$
$$=({\omega_S(A+B)}|_X\otimes\mathcal I_A)|_B=(\omega_X\otimes\mathcal I_A)|_B.$$
So the claim is proved and the previous sequence becomes
$$0\rightarrow\omega_B\rightarrow\omega_X\rightarrow{\omega_X}|_A\rightarrow 0.$$ The corresponding long exact sequence in cohomology is
$$0\rightarrow H^0(\omega_B)\rightarrow H^0(\omega_X)\rightarrow H^0({\omega_X}|_A)\rightarrow H^1(\omega_B)\rightarrow H^1(\omega_X)\rightarrow H^1({\omega_X}|_A)\rightarrow\cdots$$
Since ${\omega_X}|_A=\omega_A(D)$, by lemma \ref{a1} we have that $\dim H^1({\omega_X}|_A)=0$. Moreover, being both $B$ and $X$ connected, we have that $\dim H^1(\omega_B)=1$ and $\dim H^1(\omega_X)=1$, so the map
$H^0(\omega_X)\rightarrow H^0({\omega_X}|_A)$ is surjective.
\end{proof}

We are now able to prove proposition \ref{teo1}:
\vspace{.1cm}

\noindent{\emph {Proof of proposition \ref{teo1}.}} Let us consider the composition
\begin{equation}
H^0(\omega_X)\otimes H^0(\omega_X)\stackrel{\mu_{\omega_X}}{\longrightarrow} H^0(\omega_X^2)\stackrel{\rho_B^2}{\longrightarrow}H^0({\omega_X}^2|_B);
\end{equation}
In order to show that $\mu_{\omega_X}$ is surjective, it suffices, by a basic argument of linear algebra, to prove that
\begin{itemize}
\item[(a)]$\rho_B^2\circ \mu_{\omega_X}$ is surjective,
\item[(b)]${\rm Ker}\rho_B^2\subseteq{\rm Im}\mu_{\omega_X}$.
\end{itemize}
So let us show (a): we have a commutative diagram
\begin{equation}
\xymatrix{
H^0(\omega_X)\otimes H^0(\omega_X) \ar[r]^<<<<<<<{\rho_B^2\circ\mu_{\omega_X}} \ar[d]^{\rho_B\otimes\rho_B} & H^0({\omega_X}^2|_B)\\
H^0({\omega_X}|_B)\otimes H^0({\omega_X}|_B) \ar[ur]_{\mu_{\omega_B(D)}}}
\end{equation}
where the map $\rho_B\otimes\rho_B$ is surjective by lemma \ref{lem1} and
$\mu_{\omega_B(D)}$ is surjective by assumption (ii). So, by the commutativity of the diagram we get (a).

In order to prove (b), we notice that
$${\rm Ker\rho_B^2}=H^0(X,\mathcal{I}_B\otimes\omega_X^2),$$
and take
$$\mu:=\mu_{\omega_X}|_{H^0(X,\mathcal{I}_B\otimes\omega_X)\otimes H^0(\omega_X)}.$$
So we have the following commutative diagram:
\begin{equation}
\xymatrix{
H^0(\mathcal{I}_B\otimes\omega_X)\otimes H^0(\omega_X)\ar[r]^<<<<<{\mu}\ar[d]^{id\otimes\rho_A}& H^0(\mathcal{I}_B\otimes{\omega_X}^2)\ar[d]^\cong\\
H^0({\omega_A})\otimes H^0({\omega_X}|_A)\ar[r]^<<<<<<{\mu_{\omega_A,{\omega_X}|_A}} & H^0(\omega_A\otimes{\omega_X}|_A)
}
\end{equation}
The map $id\otimes\rho_A$ is surjective by lemma \ref{lem1}, while $\mu_{\omega_A,{\omega_X}|_A}$ is surjective by assumption (i). Hence $\mu$ is surjective. Since $\mu$ is a restriction of $\mu_{\omega_X}$, we get ${\rm Ker}\rho_B^2\subseteq{\rm Im}\mu_{\omega_X}$.\qed

\begin{Definition}
{\rm Fix an integer $m>0$; let $X$ be a reduced and Gorenstein projective curve. We say that $X$ is \textit{$m$-connected} (resp. \textit{numerically $m$-connected}) if for any decomposition $X=U\cup V$ with $U,V$ subcurves without common irreducible components, the scheme $U\cap V$ has degree at least $m$ (resp. $\deg{\omega_X}|_U-\deg\omega_U\geq m$ and $\deg{\omega_X}|_V-\deg\omega_V\geq m$).}
\end{Definition}

\begin{Remark}
{\rm If every point of $X$ lying on at least two irreducible components of $X$ is a planar singularity of $X$, then $X$ is $m$-connected if and only if it is numerically $m$-connected (see \cite{cfhr}, Remark 3.2).}
\end{Remark}

\begin{Notation}\label{X_mult}
{\rm Given a reduced curve $X$, we will denote by $X_{\rm mult}\subset X$ the set of points of $X$ lying on at least two irreducible components of $X$ and by $X_{\rm sm}$ the open set of smooth points of $X$.}
\end{Notation}

\begin{Lemma}\label{3conn}
Let $X$ be a connected, reduced and Gorenstein curve of genus $g$ with $\omega_X$ very ample. Assume that $X$ has planar singularities at the points of $X_{\rm mult}$. Then $X$ is $3$-connected.
\end{Lemma}
\begin{proof}
Let us fix any decomposition $X=U\cup V$ of $X$, with $U,V$ subcurves and $\dim (U\cap V)=0$. Set $D:=U\cap V$. Since $X$ has planar singularities at the points of $supp(D)$, $D$ is a Cartier divisor of $U$. To prove the lemma it is sufficient to show the inequality
$\deg (D) \ge 3$. Assume $\deg (D) \le 2$. Since $\omega_X$ is globally generated, $X$ is $2$-connected (see \cite{c}, Theorem D). Assume, then, $\deg D=2$. Remark \ref{properties} gives ${\omega_X}|_U\cong\omega_U(D)$. Since $X$ is $2$-connected and $\deg D=2$, we easily see that $U$ is connected. By lemma \ref{a1} we get that $\dim H^1(\omega_U(D))=0$. Thus Riemann-Roch gives $$\dim H^0(\omega_U(D))=\dim H^0(\omega_U)+1.$$
Since $D$ is a Cartier divisor of $U$, we get $\mathcal I_D\otimes\omega_U(D)\cong\omega_U$.
Thus $$\dim H^0(\mathcal I_D\otimes{\omega_X}|_U)=\dim H^0({\omega_X}|_U)-1,$$
hence the restriction to $D$ of the morphism induced by $|\omega_X|$ is not very ample, contradiction.
\end{proof}

\begin{Definition}\label{normGen}
{\rm One says that a line bundle $L$ on a curve $X$ is \textit{normally generated} if the maps
$$H^0(X,L)^k\rightarrow H^0(X,L^k)$$ are surjective for any $k\geq 1$.}
\end{Definition}

Now we need to recall Theorem B in \cite{Franciosi2}.

\begin{Theorem}[Franciosi]\label{fran}
Let $C$ be a connected reduced curve and let $\mathcal H$ be an invertible sheaf on $C$ such that
$$\deg\mathcal H|_Z\geq2p_a(Z)+1\makebox{ for all subcurves }Z\subseteq C.$$
Then $\mathcal H$ is normally generated on $C$.
\end{Theorem}

We are now able to prove the following lemma.

\begin{Lemma}\label{Bnormgen}
Let $X=A\cup B$, with $A,B\neq\emptyset$ and assume that $X$ is Gorenstein, with planar singularities at the points of $X_{\rm mult}$. Let $\omega_X$ be very ample. Then ${\omega_X}|_A$ and ${\omega_X}|_B$ are normally generated.
\end{Lemma}

\begin{proof}
Let us prove the conclusions for $B$. By Theorem \ref{fran} it is sufficient to prove that $\deg{\omega_X}|_Z\geq2p_a(Z)+1$ for every subcurve $Z\subseteq B$. Since $A\neq\emptyset$, we have that $Z\subsetneq X$. But since $\omega_X$ is very ample, by lemma \ref{3conn} we have that $X$ is $3$-connected, hence the conclusions.
\end{proof}

We are now ready to prove Theorem \ref{teorema}:
\vspace{.1cm}

\noindent{\emph {Proof of theorem \ref{teorema}}}. We recall that $X$ is a connected, reduced and Gorenstein projective curve of genus $g$ with $\omega_X$ very ample. By hypothesis we assume that $X$ has planar singularities at the points of $X_{\rm mult}$, and that $X=A\cup B$ with $A,B$ connected subcurves being smooth at $D:=A\cap B$.
Since $\mu_{\omega_A,{\omega_X}|_A}$ is surjective, by proposition \ref{teo1} it suffices to show that (ii) holds. But this is true by lemma \ref{Bnormgen}.
\qed
\vspace{.1cm}

In what follows we will investigate when condition (i) of proposition \ref{teo1} holds. If $X$ is any curve, we denote by $X_{\rm sm}$ its smooth locus. We recall a result from \cite{Ballico}; before doing this, let us introduce some notation: if $L$ is a line bundle on a curve $C$ globally generated and such that $\dim H^0(C,L)=r$, it induces a morphism $$h_L:C\rightarrow\mathbb P^{r-1}.$$
\begin{Lemma}[Ballico]\label{lem1.2}
Let $C$ be an integral projective curve with $C\neq\mathbb P^1$ and $R\in{\rm Pic}C$, $R$ globally generated and such that $h_R$ is birational onto its image. Then the multiplication map
$$\mu_{\omega_C,R}:H^0(C,\omega_C)\otimes H^0(C,R)\rightarrow H^0(C,\omega_C\otimes R)$$
is surjective.
\end{Lemma}

More in general we have the following result.

\begin{Theorem}\label{teo2}
Let $A$ be a reduced, connected and Gorenstein projective curve such that $\omega_A$ is very ample and the map $\mu_{\omega_A}$ is surjective. Let $E\subset A_{\rm sm}$ be an effective divisor on $A$ such that $\deg E\geq 2$. Then $\mu_{\omega_A,{\omega_A(E)}}$ is surjective.
\end{Theorem}

\begin{proof}
Since $A$ is connected, lemma \ref{a1} gives $H^1(\omega_A(D))=0$ for every effective and nonzero Cartier divisor $D$ on $A$. Thus
$$\dim H^0(\omega_A(D))=g_A+\deg D-1$$
for every such $D$. We use induction on $e:=\deg E$.
\vspace{0.2cm}

(a) Let us first assume $e=2$. We check that $\omega_A(E)$ is globally generated. Set $E=p_1+p_2$, where $p_1,p_2$ are smooth points for $A$. Since $\omega_A$ is globally generated, then $\omega_A(E)$ is globally generated outside $\{p_1,p_2\}$. We just proved that $$\dim H^0(\omega_A(p_i))=\dim H^0(\omega_A(p_1+p_2))-1.$$ Thus there is at least one section of $\omega_A(E)$ that doesn't vanish at $p_i$, with $i=1,2$. Hence $\omega_A(E)$ is globally generated. The divisor $E$ induces two inclusions $j:\omega_A\hookrightarrow\omega_A(E)$ and $j':{\omega_A}^2\hookrightarrow{\omega_A}^2(E)$,
which in turn induce the linear maps $j_*:H^0(\omega_A)\longrightarrow H^0(\omega_A(E))$
and $j'_*:H^0({\omega_A}^2)\longrightarrow H^0({\omega_A}^2(E))$
which have respectively corank $1$ and $2$. Consider the following diagram:
\begin{equation}
\xymatrix{
H^0(\omega_A)\otimes H^0(\omega_A)\ar[r]^<<<<<{id\otimes j_*}\ar[d]^{\mu_{\omega_A,\omega_A}}& H^0(\omega_A)\otimes H^0(\omega_A(E))\ar[d]^{\mu_{\omega_A,\omega_A(E)}}\\
H^0({\omega_A}^2)\ar[r]^<<<<<<<<<<<{j'_*} & H^0({\omega_A}^2(E))
}
\end{equation}

Since by hypothesis $\mu_{\omega_A,\omega_A}$ is surjective and $$\dim H^0({\omega_A}^2(E))=\dim H^0({\omega_A}^2)  +2,$$ then $j'_*(Im(\mu_{\omega_A,\omega_A}))$ is the codimension $2$ linear subspace $\Gamma:=H^0(\mathcal I_E\otimes\omega_A(E))$ of $H^0({\omega_A}^2(E))$. Since the subspace $j'_*(Im(\mu_{\omega_A,\omega_A}))$ is contained in $Im(\mu_{\omega_A,\omega_A(E)})$, in order to get the conclusions for $e=2$ it suffices to prove the
existence of two elements of $Im (\mu _{\omega _A,\omega _A(E)})$ which together with a basis
of $j'_*(Im(\mu_{\omega_A,\omega_A}))$, i.e. of $\Gamma$, are linearly independent. Since $\omega_A(E)$ is globally generated, there exists $\alpha\in H^0(\omega_A(E))$ not vanishing at $p_1$ and $p_2$. Since $\omega_A$ is globally generated, there is $\beta\in H^0(\omega_A)$ not vanishing at $p_1$ and $p_2$ as well. Since $\omega_A$ is very ample, there is $\gamma\in H^0(\omega_A)$ vanishing at $p_1$ but not at $p_2$, or, in the case when $p_1=p_2$, vanishing at $p_1$ with order exactly $1$. Now the section $\sigma:=\mu_{\omega_A,\omega_A(E)}(\gamma\otimes\alpha)$ doesn't belong to $\Gamma$; indeed, if $p_1\neq p_2$, $\sigma$ doesn't vanish at $p_2$, and if $p_1=p_2$, it vanishes at $p_1$ with order exactly $1$. Since the section $\mu_{\omega_A,\omega_A(E)}(\beta\otimes\alpha)$ does not vanish at $p_1$, it is not contained in the linear span of $\Gamma$ and $\sigma$. Thus $$\dim Im(\mu_{\omega_A,\omega_A(E)}) \ge dim\Gamma +2.$$ Thus $\mu_{\omega_A,\omega_A(E)}$ is surjective in the case $e=2$.

\vspace{0.2cm}

(b) Let now $e\geq 3$. We use induction on $e$. We fix a point $p$ contained in the support of the divisor $E$, and set $F:=E-p$. We check that $\omega_A(E)$ is globally generated, By inductive hypothesis the line bundle $\omega_A(F)$ is globally generated, hence so is $\omega_A(E)$ outside $p$. Since $\dim H^1(\omega_A(F))=0$, Riemann-Roch gives $\dim H^0(\omega_A(E))>\dim H^0(\omega_A(F))$. Thus $\omega_A(F)$ has a section not vanishing at $p$. Hence $\omega_A(E)$ is globally generated. We define two inclusions:
$\iota:\omega_A(F)\hookrightarrow\omega_A(E)$ and
$\iota':{\omega_A}^2(F)\hookrightarrow{\omega_A}^2(E)$,
which induce the linear maps $\iota_*:H^0(\omega_A(F))\longrightarrow H^0(\omega_A(E))$
and $\iota'_*:H^0({\omega_A}^2(F))\longrightarrow H^0({\omega_A}^2(E))$,
both having corank $1$.
We consider the diagram
\begin{equation}
\xymatrix{
H^0(\omega_A)\otimes H^0(\omega_A(F))\ar[r]^<<<<<{id\otimes u_*}\ar[d]^{\mu_{\omega_A,\omega_A(F)}}& H^0(\omega_A)\otimes H^0(\omega_A(E))\ar[d]^{\mu_{\omega_A,\omega_A(E)}}\\
H^0({\omega_A}^2(F))\ar[r]^<<<<<<<<<<<{u'_*} & H^0({\omega_A}^2(E))
}
\end{equation}

By the inductive hypothesis the map $\mu_{\omega_A,\omega_A(F)}$ is surjective. Thus the linear subspace $u'_*(Im(\mu_{\omega_A,\omega_A(F)}))$ has codimension $1$ in $H^0({\omega_A}^2(E))$. Fix $\eta\in H^0(\omega_A)$ not vanishing at $p$ and $\tau\in H^0(\omega_A(E))$ not vanishing at $p$. Since $\mu_{\omega_A,\omega_A(E)}(\eta\otimes\tau)$ does not vanish at $p$, it doesn't belong to $u'_*(Im(\mu_{\omega_A,\omega_A(F)}))$. Thus $\mu_{\omega_A,\omega_A(E)}$ is surjective.
\end{proof}

\section{$k$-normality in higher degree}
We are now interested in studying the surjectivity of higher order maps, i.e. of
$$Sym^k(H^0(\omega_X))\longrightarrow H^0(\omega_X^k)$$
when $k\geq 3$, but since $Sym^k(H^0(\omega_X))$ is a quotient of $H^0(\omega_X)^{\otimes k}$, we can equivalently study the surjectivity of
$$H^0(\omega_X)^{\otimes k}\longrightarrow H^0(\omega_X^k).$$
We observe that by applying part (b) in the proof of theorem \ref{teo2} we get the following:

\begin{Proposition}\label{prop1}
Let $A$ be a reduced, connected and Gorenstein curve such that $\omega_A$ is globally generated. Fix a globally generated $R\in{\rm Pic}A$ such that $H^1(R)=0$ and $\mu_{\omega_A,R}$ is surjective. Let $D\subset A_{\rm sm}$ be any effective divisor. Then $\mu_{\omega_A,R(D)}$ is surjective.
\end{Proposition}
As a corollary of theorem \ref{teo2}, we get the following result.

\begin{Corollary}\label{hyp(i)}
Let $A$ be a reduced, connected and Gorenstein projective curve such that $\omega_A$ is very ample and $\mu_{\omega_A}$ is surjective. Let $E\subset A_{\rm sm}$ be an effective divisor such that $\deg E\geq 2$. Then the maps $\mu_{\omega_A,{\omega_A^k(kE)}}$ are surjective for all $k\geq 2$.
\end{Corollary}

We are now going to give some definitions in order to state a result;
\begin{Definition}
{\rm A \textit{simple} $(r-1)$-secant is a configuration of $r-1$ smooth points $p_1,\ldots,p_{r-1}$ on a curve $X \subset \mathbb P^N$, spanning a $\mathbb P^{r-2}$ and such that $X\cap\mathbb P^{r-2}=\{p_1,\ldots,p_{r-1}\}$ as schemes.}
\end{Definition}

\begin{Definition}
{\rm Let $R$ be a globally generated line bundle on a curve $X$, inducing a map $h_R:X\longrightarrow\mathbb P^r$, $r:= \dim H^0(R)-1$, which is birational onto the image. A \textit{good} $(r-1)$-secant of $R$ is a set $S:=\{p_1,\ldots,p_{r-1}\}$ such that
$\dim H^0(R(-\sum_{i=1}^{r-1}p_i))=2$, $R(-\sum_{i=1}^{r-1}p_i)$ is still globally generated, and $h_R$ is an embedding at each $p_i$.}
\end{Definition}

We recall the following result from \cite{Ballico}

\begin{Lemma}[Ballico]\label{lem1.1}
Let $X$ be a one-dimensional projective locally Cohen-Macaulay scheme with $\dim H^0(\mathcal O_X)=1$ and $R\in{\rm Pic}X$ globally generated and such that $\dim H^0(R)=2$. Then the multiplication map
$$\mu_{\omega_X,R}:H^0(\omega_X)\otimes H^0(R)\longrightarrow H^0(\omega_X\otimes R)$$is surjective.
\end{Lemma}

\begin{Lemma}\label{lemmino}
Let A be a connected, projective curve, $L,M\in{\rm Pic}A$, $M$ globally generated, and such that $\dim H^0(M)=2$ and $\dim H^1(L\otimes M^\vee)=0$. Then $\mu_{L,M}$ is surjective.
\end{Lemma}
\begin{proof}
Obvious by the base point free pencil trick.
\end{proof}

\begin{Proposition}
Let $A$ be a connected, Gorenstein curve with $\omega_A$ globally generated, $R\in{\rm Pic}A$ with $R$ globally generated, with $h_R$ birational onto its image and with a good $(r-1)$-secant, where
$r:= h^0(R)-1$. Then the maps $\mu_{\omega_A,R^k}$ are surjective for all $k\geq 1$.
\end{Proposition}

\begin{proof}
Fix a good $(r-1)$-secant set $S=\{q_1,\ldots,q_{r-1}\}$. Thus the linear span $\langle h_R(q_1),\dots ,h_R(q_{r-1})\rangle$ has dimension  $r-2$, $h_R(A)\cap \langle h_R(q_1),\dots ,h_R(q_{r-1})\rangle = \{h_R(q_1),\dots ,h_R(q_{r-1})\}$ as schemes
and $$h_R^{-1}(\{h_R(q_1),\dots ,h_R(q_{r-1})\}) = \{q_1,\dots ,q_{r-1}\}.$$ Set $M:= R(-S)$. We start by examining the case $k=1$. Since $\omega_A$ is globally generated, we have $A\ne \mathbb {P}^1$. Since the map $h_R$ induced by $R$ is birational onto its image, we have $r\ge 2$. The first condition on the good $(r-1)$-secant points gives $h^0(M)=2$. The last two conditions give that $M$ is globally generated.
Since $h^0(R) = h^0(M)+r-1$, we also get $h^0(M(q_1)) = h^0(M)+1$. Thus there is $\eta \in H^0(M(q_1))$ such that $\eta (q_1)\ne 0$. The factorization shown in the following diagram
$$
\xymatrix{
H^0(\omega_A)\otimes H^0(M)\ar@{->>}[r]\ar[d]&H^0(\omega_A\otimes M)\ar[d]^j\\
H^0(\omega_A)\otimes H^0(M(q_1))\ar[r]^<<<<\varphi & H^0(\omega_A\otimes M(q_1))
}
$$

shows that the image of $\varphi$ contains a copy of $H^0(\omega_A\otimes M)$ as a hyperplane. Since  $q_1$ is not a base point for $M$ and $\omega _A$ is globally generated,
there is $\sigma \in H^0(\omega_A)\otimes H^0(M(q_1))$ that doesn't vanish on $q_1$. Hence the image of $\sigma$ via $\varphi$ doesn't vanish on $q_1$, and we get the surjectivity of $\varphi$. Repeating this argument for all the points $q_1,\ldots,q_{r-1}$ adding them one by one we get that $\mu_{\omega_A,R}$ is surjective.

Now we assume $k\geq 2$ and use induction on $k$. The inductive assumption gives the surjectivity of the map
$H^0(\omega_A)\otimes H^0(R^{k-1})\longrightarrow H^0(\omega_A\otimes R^{k-1})$.
We use the following commutative diagram:
$$
\xymatrix{
H^0(\omega_A)\otimes H^0(R^{k-1})\otimes H^0(R^k)\ar@{->>}[r]^<<<<\psi\ar[d]&H^0(\omega_A\otimes R^{k-1})\otimes H^0(R)\ar[d]^\phi\\
H^0(\omega_A)\otimes H^0(R^{k})\ar[r]^<<<<<<<<<<<<<\mu & H^0(\omega_A\otimes R^{k})
}
$$

It suffices to prove that $\phi$ is surjective, indeed, if it is, then $\phi\circ\psi$ is surjective, hence $\mu$ must be surjective. We proved that $M$ is globally generated and $\dim H^0(M)=2$. Moreover we notice that
$$\omega_A\otimes R^{k-1}\otimes M^\vee=\omega_A\otimes R^{k-2}(S).$$
Since $k\geq 2$ and $S\neq\emptyset$, we have that $\dim H^1(\omega_A\otimes R^{k-2}(S)=0$.
The base point free pencil trick applied to $\omega_A\otimes R^{k-1}$ and $M$ gives the surjectivity of $\mu_{\omega_A\otimes R^{k-1},M}$. By Riemann-Roch theorem we get that
$$\dim H^0(\omega_A\otimes R^k)=\dim H^0(\omega_A\otimes R^{k-1}\otimes M)+\sharp S.$$
Arguing as in case $k=1$ we get that the map $\mu_{\omega_A,R^k}$ is surjective.
\end{proof}

\begin{Definition}
{\rm We say that a line bundle $L$ on a curve $X$ is \textit{$k$-normally generated} if the map
$$H^0(\omega_X)^{\otimes k}\longrightarrow H^0(\omega_X^k)$$ is surjective.}
 \end{Definition}
For instance ``quadratically normal'' means ``linearly normal'' plus ``$2$-normally generated''.

\begin{Proposition}\label{kNormGen}
Let $X$ be a connected, reduced, Gorenstein projective curve with planar singularities and $\omega_X$ very ample. Assume that $X=A\cup B$, with $A,B$ connected and smooth at $D:=A\cap B$. Fix $k\geq 3$; if
\begin{itemize}
\item[(i)]$\omega_X$ is $(k-1)$-normally generated,
\item[(ii)]$\mu_{\omega_A,{\omega_X^j}|_A}$ is surjective for $1\leq j\leq k$,
\item[(iii)]${\omega_X}|_B$ is $j$-normally generated for $1\leq j\leq k$,
\end{itemize}
then
$\omega_X$ is $k$-normally generated.
\end{Proposition}

\begin{proof}
The proof is similar to the one of proposition \ref{teo1}; we just change notation slightly, denoting the multiplication maps in an easier way. We notice that in order to prove that the map
$$H^0(\omega_X)^{\otimes k}\xrightarrow{\mu_k} H^0(\omega_X^k)$$ is surjective, by factorizing we get
$$H^0(\omega_X)\otimes H^0(\omega_X)^{\otimes {k-1}}\xrightarrow{\mu\otimes\mu_{k-1}} H^0(\omega_X)\otimes H^0(\omega_X^{k-1})\xrightarrow{\widetilde\mu} H^0(\omega_X^k),$$
so it suffices to see that the map $\widetilde\mu$ is surjective. We consider the diagram
\begin{equation}\label{diag}
\xymatrix{H^0(\omega_X)\otimes H^0(\omega_X^{k-1})\ar[r]^<<<<<<<{\widetilde\mu}\ar[d]^\eta
&H^0(\omega_X^k)\ar[d]^\psi\\
H^0({\omega_X}|_B)\otimes H^0({\omega_X^{k-1}}|_B)\ar[r]^<<<<\phi&H^0({\omega_X^k}|_B)
}
\end{equation}

where the map $\widetilde\mu=\mu_{\omega_X,\omega_X^{k-1}}$.
We know that $\phi$ is surjective by (iii), and if
\begin{itemize}
\item[(a)]$\psi\circ\widetilde\mu$ is surjective,
\item[(b)]${\rm Ker}\psi\subseteq{\rm Im}\widetilde\mu$,
\end{itemize}
then by linear algebra we get that $\widetilde\mu$ is surjective. In order to prove (a), by (\ref{diag}) we equivalently show that the map $\phi\circ\eta$ is surjective. We claim that $\eta$ is surjective. Indeed, since $\omega _X$
is locally free we have the exact sequence
$$0\rightarrow\mathcal I_B\otimes\omega_X\rightarrow\omega_X\rightarrow{\omega_X}|_B\rightarrow 0.$$
If we tensor by $\omega_X^{k-2}$, we get
$$0\rightarrow\mathcal I_B\otimes\omega_X^{k-1}\rightarrow\omega_X^{k-1}\rightarrow{\omega_X}|_B\otimes\omega_X^{k-2}\rightarrow 0,$$
which is equivalent to
$$0\rightarrow\omega_A\otimes\omega_X^{k-2}\rightarrow\omega_X^{k-1}\rightarrow{\omega_X^{k-1}}|_B\rightarrow 0,$$
The corresponding long exact sequence in cohomology is
$$0\rightarrow H^0(\omega_A\otimes\omega_X^{k-2})\rightarrow H^0(\omega_X^{k-1})\rightarrow H^0({\omega_X^{k-1}}|_B)\rightarrow H^1(\omega_A\otimes\omega_X^{k-2})\rightarrow $$
$$\rightarrow H^1(\omega_X^{k-1})\rightarrow H^1({\omega_X^{k-1}}|_B)\rightarrow\cdots$$
Now we consider $H^1(\omega_A\otimes\omega_X^{k-2})$; we have that $\omega_A\otimes\omega_X^{k-2}=\omega_A\otimes{\omega_X^{k-2}}|_A=\omega_A\otimes{\omega_A^{k-2}}((k-2)D)$, hence
by lemma \ref{a1} we obtain that $H^1(\omega_A\otimes{\omega_A^{k-2}}((k-2)D))=0$, therefore the map
$$H^0(\omega_X^{k-1})\rightarrow H^0({\omega_X^{k-1}}|_B)$$ is surjective, and we get (a).

Now we want to prove (b). We notice that
$${\rm Ker}\psi=H^0(\mathcal I_B\otimes\omega_X^k)$$

and set
$$\mu:=\widetilde\mu|_{H^0(X,\mathcal{I}_B\otimes\omega_X)\otimes H^0(\omega_X^{k-1})}.$$
We have the following commutative diagram:
\begin{equation}\label{diagbis}
\xymatrix{
H^0(\mathcal{I}_B\otimes\omega_X)\otimes H^0(\omega_X^{k-1})\ar[r]^<<<<<\mu \ar[d]^\gamma
&H^0(\mathcal{I}_B\otimes{\omega_X}^k)\ar[d]^\cong\\
H^0({\omega_A})\otimes H^0({\omega_X^{k-1}}|_A)\ar[r]^<<<<<<{\mu_{\omega_A,{\omega_X^{k-1}}|_A}}
&H^0(\omega_A\otimes{\omega_X^{k-1}}|_A)}
\end{equation}

Now we have that $\mathcal{I}_B\otimes\omega_X\cong\omega_A$ and applying the previous argument to $A$ rather than to $B$, we obtain that
$$H^0(\omega_X^{k-1})\rightarrow H^0({\omega_X^{k-1}}|_A)$$ is surjective, hence so is $\gamma$ in (\ref{diagbis}). Applying hypothesis (ii) we have that $\mu$ is surjective, hence as in the proof of \ref{teo1}, we get that ${\rm Ker}\psi={\rm Im}\mu\subseteq{\rm Im}\widetilde\mu$.
\end{proof}

We notice that when $k$ grows, the hypothesis in proposition \ref{kNormGen} can be simplified:

\begin{Proposition}\label{k>3}
Let $X$ be a connected, reduced, Gorenstein projective curve of genus $g$, with $\omega_X$ globally generated. Fix $k\geq 4$ and assume that $\omega_X$ is $(k-1)$-normally generated. Then $\omega_X$ is $k$-normally generated.
\end{Proposition}
\begin{proof}
As in the proof of \ref{kNormGen}, looking at the factorization
$$H^0(\omega_X)\otimes H^0(\omega_X)^{\otimes {k-1}}\xrightarrow{\mu\otimes\mu_{k-1}} H^0(\omega_X)\otimes H^0(\omega_X^{k-1})\xrightarrow{\mu_{\omega_X,\omega_X^{k-1}}} H^0(\omega_X^k),$$
by hypothesis it suffices to prove that $\mu_{\omega_X,\omega_X^{k-1}}$ is surjective.
We use Proposition 8 in \cite{Franciosi3} in the following way:
we take $\mathcal F:=\omega_X$ and $\mathcal H:=\omega_X^{k-1}$, so we have that $H^0(\mathcal F)$ is globally generated. Moreover we have that

$$H^1(\mathcal H\otimes\mathcal F^{-1})=H^1(\omega_X^{k-2})=0$$
if $k\geq 4$, so we get that the $\mu_{\omega_X,\omega_X^{k-1}}$ is surjective.
\end{proof}

\section{Applications}

In the sequel we are going to study some cases where we can apply our results.

\begin{Lemma}\label{a2}
Let $Z$ be a connected and Gorenstein curve such that $\omega _Z$ is globally generated. Let $D \subset Z_{\rm sm}$ be an effective Cartier divisor such that
$\deg (D) \ge 2$. Then $\omega _Z(D)$ is globally generated.
\end{Lemma}

\begin{proof}
Since $\omega _Z(D)$ is a line bundle, it is globally generated if and only if for every $q\in Z$
there is $s\in H^0(\omega _Z(D))$ such that $s(q)\ne 0$. Since $\omega _Z$ is assumed to be globally generated and $D$ is effective, the sheaf $\omega _Z(D)$ is globally generated outside the finitely many points appearing in $supp(D)$. Fix $p\in supp(D)$ and set $D':= \mathcal {I}_p\otimes D$. Since $p\in X_{sm}$, $D'$ is a Cartier divisor of degree $\deg (D)-1$. Moreover, since
$p\in supp(D)$, $D'$ is effective. Thus Lemma \ref{a1} gives $h^1(\omega _Z(D'))=0$. Riemann-Roch gives $h^0(\omega _Z(D)) = h^0(\omega _Z(D'))+1$. Thus there
is $s\in H^0(\omega _Z(D))$ such that $s(p) \ne 0$.
\end{proof}

\begin{Corollary}\label{cor2comp}
Let $X$ be a connected reduced curve with two irreducible non-rational components $C_1,C_2$ meeting at planar singularities for $X$ and both smooth at $C_1\cap C_2$; assume that $\omega_X$ is very ample. Then $X$ is canonically embedded is projectively normal.
\end{Corollary}

\begin{proof}First of all we have to prove that $X$ is quadratically normal, so let us use the set-up of proposition \ref{teo1}, and set $A=C_1$, $B=C_2$. We look at hypothesis (i) and (ii) of the theorem; hypothesis (i) is verified by applying \ref{lem1.2} to $C_1$. Indeed in our situation $R={\omega_X}|_{C_1}$, i.e. $R=\omega_{C_1}(D)$ where $D$ is the divisor on $C_1$ and $C_2$ corresponding to $C_1\cap C_2$. Hence by lemma \ref{a2} we have that $R$ is globally generated and birational onto the image, and we get (i). Concerning (ii), it suffices to apply \ref{Bnormgen}, and then by \ref{teo1} we obtain that $X$ is quadratically generated. Now we want to study the $3$-normal generation of $X$. So we look at the hypothesis of \ref{kNormGen}: we know that $\omega_X$ is quadratically normal, and of course (iii) holds by lemma \ref{Bnormgen}. So it remains to prove (ii): but this is a consequence of corollary \ref{hyp(i)}, indeed we have that $\mu_{\omega_A}$ is surjective since $A$ is irreducible and hence projectively normal, moreover, being $\omega_X$ very ample, $A\cdot B\geq 3$. Now when $k\geq 4$ we just apply \ref{k>3} and get the conclusions.
\end{proof}

\begin{Remark}
{\rm We observe that in the case of nodal connected curves with two non-rational irreducible components, the corollary above says that if the two components $C_1$ and $C_2$ meet at least at $3$ points, then $X=C_1\cup C_2$ canonically embedded is projectively normal. The corollary
leaves out the curves having at least one $\mathbb P^1$ as a component, and in particular binary curves (i.e. a curve $X$ is binary if it is composed of two $\mathbb P^1$'s meeting at $g+1$ points where $g$ is the genus of $X$), but for the latter special class of curves we can use \cite{Schreyer} (see \ref{schreyer}) and easily get projective normality. Concerning the class of curves $X=C_1\cup C_2$ with $C_1\neq\mathbb P^1$ and $C_2=\mathbb P^1$, we get the projective normality by applying the same proof as in corollary \ref{cor2comp}, once we denote by $A$ the component $C_1$. Indeed the hypothesis $C_1\neq\mathbb P^1$ is used only when we apply \ref{lem1.2} to $A$.}
\end{Remark}

We can generalize the previous result:

\begin{Corollary}\label{cor3connessa}
Let $X$ be a connected reduced Gorenstein curve with $\omega_X$ very ample and with planar singularities. Assume that $X=A\cup B$ with $A\neq\mathbb P^1$ irreducible and let $B$ be a connected curve. Let $A$ and $B$ be smooth at $A\cap B$. Then $\omega_X$ is $k$-normally generated for any $k\geq 2$.
\end{Corollary}
\begin{proof}
The proof is straightforward once we notice that we can apply \ref{lem1.2} to $A$ and by Theorem \ref{teorema} we get quadratic normality of $X$; for $k=3$ we apply \ref{kNormGen} since both \ref{Bnormgen} for $B$ and \ref{lem1.2} for $A$ hold, and when $k\geq 4$ we apply \ref{k>3}.
\end{proof}

\begin{Corollary}
Let $X$ be a connected reduced Gorenstein curve with $\omega_X$ very ample and with planar singularities. Assume that $X=A\cup B$ with $A$ as in theorem \ref{teo2} and let $B$ be a connected curve.
Let $A$ and $B$ be smooth at $A\cap B$. Then $X$ canonically embedded is projectively normal.
\end{Corollary}

\begin{proof}
The proof is as in corollary \ref{cor3connessa}, we just apply theorem \ref{teo2} to $A$.
\end{proof}

We give now an example; before doing this, we recall an important result from \cite{Schreyer}:

\begin{Theorem}[Schreyer]\label{schreyer}
Let $X\subset\mathbb P^{g-1}$ be a canonical curve of genus $g$. If $X$ has a simple $(g-2)$-secant, then $X$ is projectively normal.
\end{Theorem}

Schreyer's theorem can be used in the most general setting once one is able to verify the existence of a simple $(g-2)$-secant. In \cite{Schreyer}{pp.86} gave an example of a reducible canonically embedded curve admitting no simple $(g-2)$-secant. In the following example we show that our theorem  applies to that case.
\begin{Example}\label{ex2comp}
{\rm Let $X=X_1\cup X_2\cup X_3\cup X_4$, with $X_i$ smooth of genus $g_i$ and such that the components intersect in $6$ distinct points $p_{ij}=X_i\cap X_j$ that are ordinary nodes for $X$. Then $X$ has genus $g=g_1+g_2+g_3+g_4+3$. We have that $\omega_X$ is a very ample line bundle; if $g_i=0$ for every $i$ we have a graph curve, and it is projectively normal, as we see in \cite{Graph}. Hence we can assume $g_i\neq 0$ for some $i$, say $g_1>0$. Set $A:=X_1$, $B:=X_2\cup X_3\cup X_4$. Since $A\neq\mathbb P^1$ we can apply \ref{lem1.2} and get that the multiplication map $\mu_{\omega_A,{\omega_X}|_A}$ is surjective. Sincethe conditions on the degree of ${\omega_X}|_B$ in \ref{fran} are satisfied,  the map $\mu_{{\omega_X}|_B}$ is surjective and we can apply proposition \ref{teo1} and get that $X$ is quadratically normal.}
\end{Example}

\end{document}